\documentclass[11pt,a4paper,notitlepage,oneside]{amsart}
\usepackage{amsmath,amsthm,braket}
\usepackage[new]{old-arrows}
\usepackage[psamsfonts]{amssymb}
\usepackage[T1]{fontenc}
\usepackage[italian, english]{babel}
\usepackage[utf8]{inputenc}
\usepackage{enumitem}
\usepackage{fancyhdr}
\usepackage{tikz-cd}
\usepackage{amscd}
\usepackage[a4paper, bottom=3cm, left=3cm,right=3cm]{geometry}
\usepackage{hyperref}
\theoremstyle{definition}
\newtheorem{defi}{Definition}[section]
\newtheorem{ex}[defi]{Example}
\theoremstyle{plain}
\newtheorem{teor}[defi]{Theorem}
\theoremstyle{plain}
\newtheorem{prop}[defi]{Proposition}
\theoremstyle{plain}
\newtheorem{lemma}[defi]{Lemma}
\theoremstyle{plain}
\newtheorem{cor}[defi]{Corollary}

\theoremstyle{definition}
\newtheorem{rem}[defi]{Remark}

\theoremstyle{definition}

\theoremstyle{plain}

\newcommand{\Aut}{\text{Aut}}
\DeclareMathOperator{\id}{\rm id}
\DeclareMathOperator{\deter}{\rm det}
\DeclareMathOperator{\dis}{\rm dis}
\newcommand{\N}{\mathbb{N}}
\newcommand{\Z}{\mathbb{Z}}
\newcommand{\R}{\mathbb{R}}
\newcommand{\Q}{\mathbb{Q}}
\newcommand{\C}{\mathbb{C}}

\newcommand{\lk}{\Lambda_{K3}}

\DeclareMathOperator{\NS}{\rm NS}

\DeclareMathOperator{\T}{\rm T}
\DeclareMathOperator{\FM}{\rm FM}

\DeclareMathOperator{\G}{\mathcal{G}}
\newcommand{\Pell}{\mathcal{P}}

\title{Strongly ambiguous Hilbert squares of projective K3 surfaces with Picard number one}
\author{Riccardo Zuffetti}
\email{riccardo.zuffetti@gmail.com}
\date{}
\begin{document}
\maketitle
%%%%%%%%%%%%%%%%%%%%%%%%%%
\begin{abstract}
We provide a criterion for when Hilbert squares of complex projective K3 surfaces with Picard number one are strongly ambiguous. This criterion is the same as \cite[Proposition 3.14]{dm}, but is obtained by a different method. In particular, this enables us to compute the automorphism groups of these Hilbert squares from a different point of view with respect to \cite{bcns}.
\end{abstract}
%%%%% INTRODUCTION %%%%%%%%%%%%%
\section*{Introduction}
Let $X$ be a nonsingular projective surface over $\C$. We recall that $X^{[n]}$, the \textit{Hilbert scheme of $n$ points on $X$}, is the space parametrizing 0-dimensional subschemes $Z\subset X$ of length $n$ (see for further details \cite[Deuxième Partie]{be1}, \cite{na} and \cite[Section 21]{ghj}). If $S$ is a projective K3 surface, the Hilbert square $S^{[2]}$ is an irreducible holomorphic symplectic fourfold. We say that $S^{[2]}$ is \textit{strongly ambiguous} if there exists another K3 surface $S'$ such that $S\not\cong S'$ but $S^{[2]}\cong S'^{[2]}$. We want to study the strongly ambiguous Hilbert squares of projective K3 surfaces with Picard number one. The main result is \cite[Proposition 3.14]{dm}, proved there through derived categories. In these pages we prove the same result (see Theorem \ref{teor:analogodm}) only with the glueing of lattices, the Hodge decomposition, the Torelli theorems and some properties of the ample and movable cones of $S^{[2]}$. The method here exposed is also applied to compute $\Aut(S^{[2]})$ (see Theorem \ref{teor:bcns}). This is the main result in \cite{bcns}, but proved from a different point of view. In \cite{c} one can find how to compute $\Aut(S^{[n]})$ with $n\ge2$, while in \cite{okawa} and \cite{mmy} one can find a study on birational maps between Hilbert squares.

In Sections \ref{sec:lattices}, \ref{sec:glueing} we recall some general results on lattices, while in Sections \ref{sec:K3surf}, \ref{sec:cones} we recall some properties on Hilbert squares of projective K3 surfaces and their nef and movable cones.
Let $S$ be a projective K3 surface with Picard number one, let $h$ be an ample class on $S$ with $h^2=2e$ with $e\in\N$. In Section~\ref{sec:FMK3} we recall the definition of FM partners of $S$ (see Definition \ref{defi:FMpartners}), and the number of (isomorphism classes of) FM partners of $S$ via the glueing of lattices. In Sections~\ref{sec:speranza},~\ref{sec:strongambiguity} we study a relation between the FM partners of $S$ and their Hilbert squares, in particular we want to understand for which $e$ there exist non isomorphic K3 surfaces such that their Hilbert squares are isomorphic. In Section \ref{sec:Automorphisms} we apply the previous methods to determine $\Aut(S^{[2]})$.
%%%%% ACKNOWLEDGMENTS %%%%%%%%%%%

\vspace{0.5cm}
\textbf{Acknowledgments.} We would like to thank Chiara Camere and Bert van Geemen for their ideas and clarifications, and Alberto Cattaneo and Ciaran Meachan for their suggestions on the final version of this work.
%%%%%%%%%%%%%%%%%%%%%%%%%%
\section{Lattices}\label{sec:lattices}
A $lattice$ is a pair $(\Gamma,Q)$ consisting of a free $\mathbb{Z}$-module of finite rank $\Gamma$ together with a symmetric bilinear form $Q:\Gamma\times \Gamma\rightarrow\mathbb{Z}$. Every lattice defines a \textit{quadratic form} $q(v)=v^2:=Q(v,v)$.

Given a $\mathbb{Z}$-basis $e_1,\dots,e_r$ of $\Gamma$, the \textit{Gram matrix} of $\Gamma$ (with this basis) is the $n\times n$ symmetric matrix which represents $Q$ on this basis.
The \textit{discriminant} of a lattice $(\Gamma,Q)$ is $\dis(\Gamma,Q):=\deter G$, where $G$ is a Gram matrix of that lattice.

A lattice $(\Gamma,Q)$ is said to be \textit{non degenerate} if $\dis(\Gamma,Q)\ne0$, \textit{even} if $Q(v,v)\in 2\mathbb{Z}$ for all $v\in S$, \textit{unimodular} if $\dis(\Gamma,Q)=\pm1$.

We say that a map $f:\Gamma_1\rightarrow \Gamma_2$ is an \textit{isometry} between lattices if it is an isomorphism of $\Z$-modules and preserves the bilinear forms.

A \textit{sublattice} of a lattice $(\Gamma,Q)$ is a free submodule $\Gamma'$ of $\Gamma$ with the induced bilinear form $(\Gamma',Q|_{\Gamma'\times \Gamma'})$. If $\Gamma/\Gamma'$ is a finite group then
\begin{equation}\label{eq:spultim}
\dis(\Gamma')=[\Gamma:\Gamma']^2\cdot\dis(\Gamma).
\end{equation}

A sublattice is said to be \textit{primitive} if the quotient module $\Gamma/\Gamma'$ is a free module.
\begin{ex} Let $V$ be a subset of $\Gamma$ and consider
\[
V^\perp :=\{v'\in \Gamma : Q(v',v)=0 \quad\forall v\in V\},
\]
then $V^\perp$ is a primitive sublattice of $\Gamma$.
\end{ex}
The \textit{dual lattice} of a non degenerate lattice $(\Gamma,Q)$ is the $\Z$-module
\[
\Gamma^*:=\set{v\in \Gamma_\Q:Q(v,w)\in\Z\,\,\,\,\forall w\in\Gamma}\subset\Gamma_\Q:=\Gamma\otimes\Q,
\]
together with the extension of $Q$ on $\Gamma^*$, that is denoted again by $Q$.
It is well known that $[\Gamma^*:\Gamma]=|\det(G)|=|\dis(\Gamma,Q)|$, in particular if $\Gamma$ is unimodular then $\Gamma\cong\Gamma^*$ as lattices.

Let $(\Gamma,Q)$ be a non degenerate lattice, then the \textit{discriminant group of} $\Gamma$ is the finite group
\[
A_\Gamma:=\Gamma^*/\Gamma,
\]
which has order $[\Gamma^*:\Gamma]=|\dis(\Gamma,Q)|$. When $\Gamma$ is an even lattice, there is a well defined quadratic form $q_{\Gamma}$ on the discriminant group given by the quadratic form $q$:
\[
q_\Gamma:A_\Gamma\rightarrow \Q/2\Z,\qquad q_\Gamma(v):=Q(v,v)\text{ mod }2\Z.
\]
\begin{prop}[see {\cite[Chapter I, Lemma 2.5]{bpv}}]\label{prop:primunim}
For a primitive sublattice $E$ of a unimodular lattice $S$ there is an isomorphism $A_E\rightarrow A_{E^\perp}$.
\end{prop}
\begin{defi}
The group of self-isometries of a lattice $(\Gamma,Q)$ is denoted by $O(\Gamma)$; an element of $O(\Gamma)$ is called an isometry (or automorphism) of $\Gamma$.
\end{defi}
An automorphism $M\in O(\Gamma)$ of the lattice extends $\Q$-linearly to an automorphism of the $\Q$-vector space $\Gamma_\Q$ which preserves the $\Q$-bilinear extension of $Q$. Therefore we have a homomorphism
\[
O(\Gamma)\rightarrow O(A_\Gamma),\qquad M\rightarrow \overline{M},
\]
where $\overline{M}$ is the induced action on the discriminant group.
%%%%%%%%%%%%%%%%%%%%%%%%%%
\section{Glueing lattices and glueing automorphisms}\label{sec:glueing}
Let $(\Gamma,Q)$ be an even lattice, $E_1\subset \Gamma$ a primitive sublattice and $E_2:=E_1^{\perp}$. It is obvious that $E_1\oplus E_2\subseteq \Gamma\subseteq E_1^*\oplus E_2^*$.
\begin{rem}\label{rem:quadformortgen}
As $\Gamma$ is even and $E_1,E_2$ are perpendicular, if $s\in\Gamma$ and writing $s=(s_1,s_2)\in E_1^*\oplus E_2^*$ we get:
\[
Q(s,s)=Q(s_1,s_1)+Q(s_2,s_2)\in2\Z.
\]
So by the definition of the quadratic form on a discriminant group, we have that $q_{E_1}(s_1)=-q_{E_2}(s_2)$.
\end{rem}
\begin{rem}\label{rem:costruzioneisotropo}
Consider the abelian group $A_{E_1}\times A_{E_2}=(E_1^*\oplus E_2^*)/(E_1\oplus E_2)$ with quadratic form $q_{E_1}+q_{E_2}$. The lattice $\Gamma$ defines a subgroup
\[
I_\Gamma:=\Gamma/(E_1\oplus E_2)=\set{(s_1,s_2)\in A_{E_1}\times A_{E_2}:s\in \Gamma},
\]
that is isotropic (i.e.\ $(q_{E_1}+q_{E_2})|_{I_{\Gamma}}=0$), by Remark~\ref{rem:quadformortgen}, of cardinality $\#I_\Gamma=[\Gamma:E_1\oplus E_2]$.

Moreover, if $\Gamma$ is unimodular then $\Gamma/(E_1\oplus E_2)\cong A_{E_1}\cong A_{E_2}$ by Proposition~\ref{prop:primunim}. So there exists an isomorphism $\varphi:A_{E_1}\rightarrow A_{E_2}$ such that $q_{E_1}(a_1)=-q_{E_2}(\varphi(a_1))$ for all $a_1\in A_{E_1}$, and we simply write $q_{E_1}=-q_{E_2}$ in this case.
\end{rem}
\begin{defi}
Given two non degenerate lattices $(E_1,Q_1),(E_2,Q_2)$, we say that $(\Gamma,Q)$ arises from \textit{glueing} $(E_1,Q_1)$ and $(E_2,Q_2)$ if it is a finite index overlattice of $(E_1\oplus E_2,Q_1+Q_2)$.
\end{defi}
The following well known proposition gives a relation between glued lattices and isotropic subgroups.
\begin{prop}\label{prop:bijectionisot}
Let $(E_1,Q_1),(E_2,Q_2)$ be even non degenerate lattices.
There is a bijection between:
\begin{enumerate}
\item isotropic subgroups $I\subset A_{E_1}\times A_{E_2}$,
\item even lattices $(\Gamma,Q)$ obtained by glueing $(E_1,Q_1)$ and $(E_2,Q_2)$,
\end{enumerate}
given by
\begin{align*}
I&\mapsto \Gamma_I:=(\pi_1\times\pi_2)^{-1}(I),\\
\Gamma&\mapsto I_\Gamma:=\Gamma/(E_1\oplus E_2),
\end{align*}
where $\pi_i:E_i^*\rightarrow A_{E_i}$ for $i=1,2$ are the quotient maps.

Moreover if $A_{E_1}\cong A_{E_2}$ and $q_{E_1}=-q_{E_2}$ then the unimodular even lattices obtained by glueing $E_1$ with $E_2$ are in bijection with the (maximal) isotropic subgroups defined as $I_{\varphi}=\set{(x,\varphi(x)):x\in A_{E_1}}$ for suitable isomorphisms $\varphi:A_{E_1}\rightarrow A_{E_2}$ such that $q_{E_1}(x)=-q_{E_2}(\varphi(x))$ for all $x\in A_{E_1}$.
\end{prop}
\begin{proof}
Let $(\Gamma,Q)$ be a lattice such that $E_1\oplus E_2\subset \Gamma$, with finite index and such that $Q$ restricts to $Q_1+Q_2$ on $E_1\oplus E_2$, thus $E_2=E_1^\perp$ in $\Gamma$.
By Remark~\ref{rem:costruzioneisotropo}, $I_\Gamma$ is an isotropic subgroup of $A_{E_1}\times A_{E_2}$.

Every isotropic subgroup of $A_{E_1}\times A_{E_2}$ defines an even lattice $(\Gamma,Q)$ containing the $E_i$ as orthogonal sublattices. Indeed, let $I$ be an isotropic subgroup of $A_{E_1}\times A_{E_2}=E_1^*/E_1\times E_2^*/E_2$. Let $\Gamma_I:=(\pi_1\times\pi_2)^{-1}(I)$ and let $Q$ be the restriction to $\Gamma_I$ of the $\Q$-linear extension of $Q_1+Q_2$ to $E_1^*\times E_2^*$.
The isotropy condition shows that $Q(s,s)\in2\Z$ for all $s\in \Gamma_I$.
Moreover $Q(x,y)\in\Z$ for all $x,y\in \Gamma_I$, indeed:
\[
2\Z\ni Q(x+y,x+y)=Q(x,x)+Q(y,y)+2Q(x,y)
\]
and, since $Q$ is even, $Q(x,y)\in\Z$. Hence $(\Gamma_I,Q)$ is an even lattice.
Since $0\in I_\Gamma$, $\pi^{-1}(0)=E_1\oplus E_2\subseteq \Gamma_I$ and $E_1$, $E_2$ are obviously orthogonal.

Eventually, we prove the assertion about unimodular even lattices.
Suppose that $\Gamma$ is unimodular, then the claim follows by Remark~\ref{rem:costruzioneisotropo}.
Conversely, let $\varphi:A_{E_1}\rightarrow A_{E_2}$ be an isomorphism such that $q_{E_1}(a_1)=-q_{E_2}(\varphi(a_1))$. Let $I:=\set{(x,\varphi(x)):x\in A_{E_1}}\subset A_{E_1}\times A_{E_2}$, this is obviously isotropic and $\#I=\#A_{E_1}$. We know, see Formula~\eqref{eq:spultim}, that
\[
\dis(E_1\oplus E_2)=[\Gamma_I:E_1\oplus E_2]^2\cdot\dis(\Gamma).
\]
But $|\dis(E_1\oplus E_2)|=|\dis(E_1)|\cdot|\dis(E_2)|=(\#A_{E_1})^2$, hence $\dis(\Gamma)=\pm1$.
\end{proof}
We can also glue automorphisms on lattices obtained by glueing.
\begin{cor}\label{cor:esteniso}
Let $(E_1,Q_1),(E_2,Q_2)$ be two even non degenerate lattices and let $I,J\subset A_{E_1}\times A_{E_2}$ be two isotropic subgroups.

If $S_I,S_J$ are two overlattices of $E_1\oplus E_2$ defined by two isotropic subgroups $I,J$ and $M:S_I\rightarrow S_J$ is an isometry such that $M(E_i)\subseteq E_i$ for $i=1,2$ then $(\overline{M_1},\overline{M_2})(I)=J$, where $M_1:=M|_{E_1}$ and $M_2:=M|_{E_2}$.

Conversely, let $M_i\in O(E_i)$ for $i=1,2$ such that $(\overline{M_1},\overline{M_2})(I)=J$. Then $(M_1,M_2)$ extends to an isometry $M:S_I\rightarrow S_J$.
\end{cor}
%%%%%%%%%%%%%%%%%%%%%%%%%%
\section{Hilbert squares of projective K3 surfaces}\label{sec:K3surf}
In this section we recall some useful results on Hilbert squares of projective K3 surfaces which we will use in the following sections.

In \cite[Chapter VIII]{bpv} one can find basic properties on K3 surfaces.
We recall, to fix the notation, that if $S$ is a K3 surface then $H^2(S,\Z)$, endowed with the intersection product, is a lattice with signature $(3,19)$ isometric to $\lk:=U\oplus U\oplus U\oplus E_8(-1)\oplus E_8(-1)$, where $U$ is the hyperbolic plane and $E_8$ is the root lattice with rank 8.
\begin{defi}
Let $X$ be a nonsingular projective surface over $\C$. We define $X^{[n]}$, called \textit{Hilbert scheme of $n$ points on $X$}, as the space parametrizing 0-dimensional subschemes $Z\subset X$ of length $n$, i.e.\ $h^0(\mathcal{O}_Z)=n$ where $H^0(\mathcal{O}_Z)$ is the space of sections of the sheaf $\mathcal{O}_Z$.
\end{defi}
If $S$ is a projective K3 surface then $S^{[2]}$ is an irreducible holomorphic symplectic fourfold (see \cite[Section 21.2]{ghj}). Moreover, the second cohomology group $H^2(S^{[2]},\Z)$ carries a natural integral primitive quadratic form $q_{BB}:H^2(S^{[2]},\Z)\rightarrow\Z$, called the \textit{Beauville--Bogomolov form} (see for further details \cite[Section 4.2]{og} or \cite[Section 23]{ghj}).
\begin{prop}[see {\cite[\S6 Proposition 6]{be1}} and {\cite[Example 23.19]{ghj}}]\label{prop:svss2}
Let $S$ be a projective K3 surface. There exists an injective homomorphism $i:H^2(S,\Z)\hookrightarrow H^2(S^{[2]},\Z)$ which is compatible with the Hodge structures, and moreover
\[
H^2(S^{[2]},\Z)=i(H^2(S,\Z))\oplus\Z\xi,
\]
for a suitable $\xi\in H^2(S^{[2]},\Z)$ such that $\xi^2:=q_{BB}(\xi)=-2$.
\end{prop}
\begin{defi}\label{defi:defo}
Every smooth projective hyperkähler fourfold $F$ is called \textit{of $K3^{[2]}$-type} if there exists a projective K3 surface $S$ such that $F$ is a deformation of $S^{[2]}$.
\end{defi}
By the Torelli Theorem, two projective K3 surfaces $S$ and $S'$ are isomorphic if and only if there exists a Hodge isometry (i.e.\ an isometry which preserves the Hodge decompositions) between $H^2(S,\Z)$ and $H^2(S',\Z)$ that is effective (i.e.\ it maps ample classes in ample classes). The following theorem generalizes the previous statement to (smooth projective) hyperkähler fourfolds of $K3^{[2]}$-type. We will call it the \textit{Generalized Torelli Theorem}.
\begin{teor}[Verbitsky, Markman,~{\cite[Theorem 2.2]{dm}}]\label{teor:torelligen}
Let $F_1$ and $F_2$ be projective hyperkähler fourfolds of $K3^{[2]}$-type, with $h_1$ and $h_2$ ample classes of $F_1$ and $F_2$ respectively. Let $\Phi:H^2(F_1,\Z)\rightarrow H^2(F_2,\Z)$ be a Hodge isometry such that $\Phi(h_1)=h_2$; then there is an isomorphism $\varphi:F_2\rightarrow F_1$ such that $\Phi=\varphi^*$.
\end{teor}
\begin{rem}
Notice that, with the notation of Theorem \ref{teor:torelligen}, if $\Phi$ is a pullback of an isomorphism then it is obviously effective.
\end{rem}
%%%%%%%%%%%%%%%%%%%%%%%%%%
\section{Pell-type equations, nef and movable cones of $S^{[2]}$}\label{sec:cones}
In this section we fix the notation on Pell-type equations and we recall how the solutions of this kind of equations are useful to understand the geometry of the nef and movable cones of the Hilbert squares (see Theorem \ref{teor:bayermacri}).
\begin{defi}
The \textit{Pell-type equation} $\mathcal{P}_d(m)$ is
\[
x^2-dy^2=m,
\]
where $d,m$ are non zero integers with $d>0$.

A solution $(U,V)\in\Z^2$ of $\Pell_d(m)$ is called $positive$ if $U>0$ and $V>0$, $minimal$ if it is positive and $U$ is as small as possible.
\end{defi}
\begin{lemma}[see~{\cite[Section 40, Lemma 14]{jon}}]\label{rem:pellcolmeno}
Assume $d$ is not a square. Let $(U,V)$ be the (positive) minimal solution of the equation $\mathcal{P}_d(1)$ and put $\varepsilon:=U+V\sqrt{d}$. The integral solutions of $\mathcal{P}_d(1)$ are obtained as integral powers of $\varepsilon$, in the sense that if $\pm\varepsilon^n=\frac{1}{2}(u+v\sqrt{d})$ with $n\in\Z$, $(u,v)$ is a new solution and all solutions can be obtained in this way.

Moreover, if $\mathcal{P}_d(-1)$ has a minimal solution $(U',V')$, then, putting $\eta:=U'+V'\sqrt{d}$, the integral solutions of $\Pell_d(1)$ are $\pm\eta^{2n}$ with $n\in\Z$.
\end{lemma}
Let $X$ be a projective manifold, $\NS(X):=H^{1,1}(X)\cap H^2(X,\Z)$ is its \textit{Néron--Severi group} and $\T_X:=\NS(X)^{\perp}\subset H^2(X,\Z)$ is its \textit{transcendental lattice}.
The \textit{nef cone of $X$} is defined as the set $\text{Nef}(X)$ of all classes $\alpha\in \NS(X)\otimes\R$ with $\alpha\cdot C\ge 0$ for all curves $C\subset X$, while the \textit{movable cone} $\text{Mov}(X)$ is the cone generated by the movable classes (i.e.\ the classes of divisors $L$ on $X$ such that the base locus of $|L|$ has codimension $\ge2$).
The nef and the movable cones of $X$ are convex (see \cite[Chapter 8, Section 1]{hu2} and \cite[Section 6.5]{ma}).

Let $S$ be a K3 surface such that $\NS(S)=\Z h$ with $h$ ample and $h^2=2e$. We know that $H^2(S^{[2]},\Z)$ is a lattice with the Beauville--Bogomolov quadratic form $q_{BB}$, and it is naturally isometric to $H^2(S,\Z)\oplus\Z\xi$ (see Proposition \ref{prop:svss2}). We denote again by $h$ the class induced by $h$ in $H^2(S^{[2]},\Z)$.
\begin{rem}\label{rem:NShilbert}
The Néron--Severi of $S^{[2]}$ is $\NS(S^{[2]})=\Z h\oplus\Z\xi$.
\end{rem}
We will see now some results on the two-dimensional vector space $\NS(S^{[2]})\otimes\R$.
Every non trivial convex cone in $\R^2$ has obviously two boundary walls. Hence, if the Picard number of $S^{[2]}$ is two, there are two boundary walls for every non trivial convex cone $C$ in $\NS(S^{[2]})\otimes\R$.

The class $h$ is nef, not ample, and spans one of the two boundary walls of the nef cone $\text{Nef}(S^{[2]})\subset \NS(S^{[2]})\otimes\R$ (see~{\cite[Section 3.2]{dm}}). The other boundary wall of $\text{Nef}(S^{[2]})$ is spanned by a class $h-\nu_S\xi$, where $\nu_S$ is a positive real number. Similarly, the boundary walls of the movable cone $\text{Mov}(S^{[2]})$ are spanned by $h$ and $h-\mu_S\xi$.

The following result shows that $\nu_S$ and $\mu_S$ are rational numbers and only depend on the positive integer $e$.
\begin{teor}[Bayer--Macrì, {\cite[Section 13]{bm}}, or {\cite[Theorem 3.4]{dm}}]\label{teor:bayermacri}
Let $S$ be a K3 surface such that $\NS(S)=\Z h$ with $h$ ample and $h^2=2e$.
The slopes $\nu_S$ and $\mu_S$ are respectively equal to the rational numbers $\nu_e$ and $\mu_e$ defined as follows.
\begin{itemize}
	\item Assume that the equation $\mathcal{P}_{4e}(5)$ has no solutions.
	\begin{enumerate}
		\item If $e$ is a perfect square, we have $\nu_e:=\sqrt{e}$ and $\mu_e:=\sqrt{e}$.
		\item If $e$ is not a perfect square and $(a_1,b_1)$ is the minimal solution of the equation $\mathcal{P}_e(1)$, we have $\nu_e:=e\frac{b_1}{a_1}$ and $\mu_e:=e\frac{b_1}{a_1}$.
	\end{enumerate}
	\item Assume that the equation $\mathcal{P}_{4e}(5)$ has a solution.
	\begin{enumerate}
		\item If $(a_5,b_5)$ is its minimal solution, we have $\nu_e:=2e\frac{b_5}{a_5}$ and $\mu_e:=e\frac{b_1}{a_1}>\nu_e$.
	\end{enumerate}
\end{itemize}
\end{teor}
%%%%%%%%%%%%%%%%%%%%%%%%%%
\section{Fourier--Mukai partners of K3 surfaces}\label{sec:FMK3}
Let $S$ be a K3 surface. From now on if $V$ is a sublattice of $H^2(S,\Z)$ and $p,q\ge0$ such that $p+q=2$ then $V^{p,q}:=(V\otimes\C)\cap H^{p,q}(S)$. The same holds for $S^{[2]}$.
\begin{defi}\label{defi:FMpartners}
Let $S,S'$ be two K3 surfaces and $\T_S,\T_{S'}$ their transcendental lattices. We say that $S$ and $S'$ are \textit{Fourier--Mukai partners (or FM partners)} if there exists a Hodge isometry $\varphi:\T_S\rightarrow \T_{S'}$, i.e.\ an isometry $\varphi$ such that $\varphi_\C\big(\T_S^{p,q}\big)=\T_{S'}^{p,q}$.
\end{defi}
\begin{lemma}\label{lemma:isocap4end}
Let $S,S'$ be two K3 surfaces and $\T_S,\T_{S'}$ their transcendental lattices. An isometry $\varphi:\T_S\rightarrow \T_{S'}$ is a Hodge isometry if and only if $\varphi_\C\big(H^{2,0}(S)\big)=H^{2,0}(S')$.

In particular, if $f:H^2(S,\Z)\rightarrow H^2(S',\Z)$ is an isometry such that $f|_{\T_S}:\T_S\rightarrow \T_{S'}$ is a Hodge isometry, then $f$ is a Hodge isometry.
\end{lemma}
\begin{proof}
It follows trivially from the fact that $H^{2,0}(S)\perp H^{0,2}(S)$, $H^{1,1}(S)=(H^{2,0}(S)\oplus H^{0,2}(S))^\perp$ and $\overline{H^{2,0}(S)}=H^{0,2}(S)$.
\end{proof}
\begin{rem}\label{rem:extensionsquares}
In Lemma~\ref{lemma:isocap4end}, if we replace $S,S'$ with $S^{[2]},S'^{[2]}$ respectively, where $S,S'$ are projective K3 surfaces, an analogous result holds.
\end{rem}
Let $S$ be a K3 surface with Picard number one, such that $\NS(S)=\Z h$ with $h$ ample and $h^2=2e$.
Let $\T:=\T_S$ be the transcendental lattice of $S$ and let $t$ be an element in $\T$ such that $A_{\T}=\Braket{t/2e}$ and $t^2=-h^2$.
\begin{rem}\label{rem:hodgedevasto}
The only Hodge isometries $\varphi:\T\rightarrow\T$ are $\pm\id$. In other words $O_{\text{Hodge}}(\T)=\set{\pm\id}$ (see~\cite[Section 3.3, Corollary 3.5]{hu2}).
\end{rem}
We define
\[
\text{FM}(S):=\set{\text{isomorphism classes }[S']:\T_{S'}\underset{\text{Hodge}}{\cong}\T}.
\]
The following result is well known (see~{\cite[Proposition 1.10]{ogu}}). We provide a proof of this result using the glueing of lattices, since we will use the same idea in Section~\ref{sec:speranza}.
\begin{prop}[Oguiso]\label{prop:pen-ultimo}
Let $S$ be as above, then $\#\FM(S)=2^{p(e)-1}$, where $p(1)=1$ and $p(e)$, with $e\ge2$, is the number of primes $q$ such that $q|e$.
\end{prop}
\begin{proof}
Let $\pi_{\T}:\T^*\rightarrow A_{\T},\pi_{\Z h}:(\Z h)^*\rightarrow A_{\Z h}$ be the projection maps.
By Proposition \ref{prop:bijectionisot}, every even unimodular overlattice of $\T\oplus\Z h$ is of the form $\Gamma_I:=(\pi_{\T}\times\pi_{\Z h})^{-1}(I)$ for an isotropic subgroup $I\subset A_{\T}\times A_{\Z h}$. Note that $\Gamma_I\cong\lk$ as lattices, indeed they are even, indefinite and unimodular, with signature $(3,19)$, see \cite[Theorem 6, pag.\ 54]{se}. We fix a Hodge structure on $\Gamma_I$ by $\Gamma_I^{2,0}:=\T^{2,0}$. By the surjectivity of the period map, there exists a K3 surface $S_I$ such that $H^2(S_I,\Z)\cong\Gamma_I$ as Hodge structures. By the Weak Torelli Theorem (see \cite[Corollary 11.2]{bpv}) this surface is unique up to isomorphism.

Obviously such an $S_I$ is a FM partner of $S$, indeed every $S_I$ has transcendental lattice $\T$. Hence the problem is equivalent to counting all the even unimodular lattices $\Gamma$ obtained by the glueing of $\T$ and $\Z h$.

We know (see Proposition \ref{prop:primunim}) that $A_{\T}\cong A_{\Z h}$, moreover
\[
A_{\T}\times A_{\Z h}=\Braket{\frac{1}{2e}\Big(t,0\Big),\frac{1}{2e}\Big(0,h\Big)}.
\]
Let $I\subset A_{\T}\times A_{\Z h}$ be an isotropic subgroup that corresponds to an even unimodular overlattice $\Gamma_I$. The projections onto $A_{\T}$ and $A_{\Z h}$ are isomorphisms, so there is a unique $a\in(\Z/2e\Z)^*$ such that
\begin{equation}\label{eq:isotropohodge}
I:=I_a:=\Braket{\Big(\frac{1}{2e}t,\frac{a}{2e}h\Big)}\subset A_{\T}\times A_{\Z h}.
\end{equation}
Since $I$ has to be isotropic, we have $a^2\equiv1$ mod $4e$.

To solve our problem we have to determine when an isotropic subgroup $I_b$ defines a K3 surface $S_{I_b}$ such that $S_{I_a}\cong S_{I_b}$. For this purpose, we note that $H^2(S_{I_a},\Z)\cong H^2(S_{I_b},\Z)$ if and only if there exists a Hodge isometry $\varphi:\Gamma_{I_a}\rightarrow\Gamma_{I_b}$. Note that $\varphi|_T\in O_{\text{Hodge}}(\T)=\set{\pm\id}$ by Remark~\ref{rem:hodgedevasto} and $\varphi|_{\Z h}\in  O(\Z h)=\set{\pm\id}$, thus such a Hodge isometry $\varphi$ arises by glueing $\pm\id_{\T}$ with $\pm\id_{\Z h}$ (see Lemma \ref{lemma:isocap4end}). By Corollary \ref{cor:esteniso}, this is equivalent to either $I_a=I_b$ or $I_a=(\id,-\id)(I_b)$, i.e.\ $b=\pm a$.

It is easy to compute that $\#\set{a\in\Z/2e\Z:a^2\equiv1\,\,\text{mod }4e}=2^{p(e)}$, hence the number of isomorphism classes of FM partners of $S$ is $2^{p(e)}/2=2^{p(e)-1}$. The computation follows by the Chinese Remainder Theorem and the following well known facts:
\begin{enumerate}
	\item $\#((\Z/2\Z)^*)_2=1$,
	\item $\#((\Z/4\Z)^*)_2=2$,
	\item $\#((\Z/2^k\Z)^*)_2=4$ for all $k\ge3$,
	\item $\#((\Z/p^k\Z)^*)_2=2$ for all $p\ne2$ and $k\ge1$,
\end{enumerate}
where, given an abelian group $H$, $H_n$ is the subgroup of the $n$-torsion elements. Moreover, notice that if $4e=2^{k_0}p_1^{k_1}\dots p_n^{k_n}$ is the prime decomposition of $4e$ then
\[
\Big(\frac{\Z}{4e\Z}\Big)^*_2=\Big(\frac{\Z}{2^{k_0}\Z}\Big)^*_2\times\Big(\frac{\Z}{p_1^{k_1}\Z}\Big)^*_2\times\dots\times\Big(\frac{\Z}{p_n^{k_n}\Z}\Big)^*_2.\qedhere
\]
\end{proof}
%%%%%%%%%%%%%%%%%%%%%%%%%%
\section{Glueing Hodge isometries}\label{sec:speranza}
\begin{defi}
Let $S$ be a K3 surface (not necessarily with Picard number one). We say that $S^{[2]}$ is \textit{ambiguous} if there exists a K3 surface $S'$ and an isomorphism $S^{[2]}\rightarrow S'^{[2]}$ which is not induced by any isomorphism $S\rightarrow S'$.

We say that $S^{[2]}$ is \textit{strongly ambiguous} if there exists such a K3 surface $S'$ which is in addition not isomorphic to $S$.
\end{defi}

We want to understand when there is strong ambiguity of Hilbert squares of projective K3 surfaces with Picard number one.
For this purpose, we follow the argument of Section~\ref{sec:FMK3}.
As before, let $S$ be a K3 surface with Picard number one, such that $\NS(S)=\Z h$ with $h$ ample and $h^2=2e$.
Let $\T=\T_S$ be the transcendental lattice of $S$ and let $t\in \T$ such that $A_{\T}=\Braket{\frac{1}{2e}t}$ and $t^2=-h^2$.
The following result of Ploog enables us to consider only the Hilbert squares of FM partners of $S$ to study if $S^{[2]}$ is strongly ambiguous.
\begin{teor}[see {\cite[Proposition 10]{p}}]\label{teor:isoFM}
Let $S,S'$ be projective K3 surfaces. If $S^{[2]}\cong S'^{[2]}$ then $S$ and $S'$ are FM partners.
\end{teor}
\begin{proof}
Let $f:S^{[2]}\rightarrow S'^{[2]}$ be an isomorphism. The map $f^*:H^2(S'^{[2]},\Z)\rightarrow H^2(S^{[2]},\Z)$ is a Hodge isometry. Hence $f^*$ maps the Néron--Severi group of $S'^{[2]}$ to the Néron--Severi group of $S^{[2]}$ isomorphically. By the properties of the Beauville--Bogomolov form on Hilbert schemes (see Section~\ref{sec:K3surf}) it follows that $\T_S\cong\NS(S^{[2]})^{\perp}$ and $\T_{S'}\cong\NS(S'^{[2]})^\perp$. But $f^*$ is an isometry, hence $f^*|_{\T_{S'}}:\T_{S'}\rightarrow \T_{S}$ is a well defined Hodge isometry.
\end{proof}
Since $\det(\lk\oplus\Braket{-2})=2$, the lattice $H^2(S,\Z)\oplus\Z\xi\cong\lk\oplus\Braket{-2}$ is not a unimodular lattice.

By Remark \ref{rem:NShilbert}, we get $A_{\NS(S^{[2]})}=A_{\Z h\oplus\Braket{-2}}=A_{\Z h}\times A_{\Braket{-2}}$ and
\[
A_{\T}\times A_{\Z h}\times A_{\Braket{-2}}
=\Braket{
\begin{pmatrix}
\frac{1}{2e}t,0,0
\end{pmatrix},
\begin{pmatrix}
0,\frac{1}{2e}h,0
\end{pmatrix},
\begin{pmatrix}
0,0,\frac{1}{2}\xi
\end{pmatrix}
}.
\]
\begin{lemma}\label{lemma:specialisot}
Let $J\subset A_{\T}\times A_{\Z h}\times A_{\Braket{-2}}$ be an isotropic subgroup which defines an even overlattice $\T\oplus \NS(S^{[2]})\hookrightarrow\Gamma_J$ with $\Gamma_J\cong\lk\oplus\Braket{-2}$ and $\T\subseteq\lk$, then $\#J=2e=\#A_{\T}$. Moreover, $J=\Braket{\begin{pmatrix}
\frac{1}{2e}t,\frac{a}{2e}h,\frac{z}{2}\xi
\end{pmatrix}}$ for suitable $a\in \Z/2e\Z$ and $z\in\Z/2\Z$.
\end{lemma}
\begin{proof}
We know that $\#J=[\Gamma_J:(\T\oplus \NS(S^{[2]}))]$, thus (see Formula \ref{eq:spultim})
\[
\dis(\T\oplus \NS(S^{[2]}))=(\#J)^2\cdot\dis(\Gamma_J),
\]
but $\dis(\Gamma_J)=2$ and $\dis(\T\oplus \NS(S^{[2]}))=\dis(\T)\cdot\dis(\Z h)\cdot\dis(\Z\xi)=(-2e)\cdot 2e \cdot(-2)=8e^2$. Hence $\#J=2e$.

We know that $\T\subset\lk\hookrightarrow\lk\oplus\Braket{-2}$, hence the projections $\lk^*\oplus\Braket{-2}^*\rightarrow A_{\T}$, $\lk^*\oplus\Braket{-2}^*\rightarrow J$ are surjective. It follows that also the projection $p_{\T}:J\rightarrow A_{\T}$ is surjective.
In particular $\#J=\#A_{\T}$, so $p_{\T}$ is an isomorphism.
As $A_{\T}\cong\Z/2e\Z$ is cyclic, also $J$ is cyclic and it is generated by $\begin{pmatrix}
\frac{1}{2e}t,\frac{a}{2e}h,\frac{b}{2}\xi
\end{pmatrix}$ for suitable $a\in\Z/2e\Z$ and $z\in\Z/2\Z$, as in the claim.
\end{proof}
The isotropy condition for $J$ as in Lemma \ref{lemma:specialisot} is $-1/2e+a^2/2e-z^2/2\in2\Z$, that is equivalent to $-1+a^2-ez^2\in4e\Z$. Hence $a^2-ez^2\equiv1$ $\text{mod }4e$. We have two cases:
\begin{itemize}
\item{Case $z=0$}

We have $a^2\equiv1$ mod $4e$, as in Section~\ref{sec:FMK3}. In particular there exists a trivial bijection
\[
\{J_a:=\Braket{\begin{pmatrix}
\frac{1}{2e}t,\frac{a}{2e}h,0
\end{pmatrix}}\}\leftrightarrow\{I_a:=\Braket{\begin{pmatrix}
\frac{1}{2e}t,\frac{a}{2e}h
\end{pmatrix}}\},
\]
where $I_a\hookrightarrow A_{\T}\times A_{\Z h}$ is as in \eqref{eq:isotropohodge}.
\item{Case $z=1$}

We have $a^2\equiv1+e$ mod $4e$.
\end{itemize}
The following result will allow us to consider only the case $z=0$.
\begin{lemma}\label{lemma:stragrande}
With the notation of Lemma \ref{lemma:specialisot}, let $J_a:=\Braket{\begin{pmatrix}
\frac{1}{2e}t,\frac{a}{2e}h,0
\end{pmatrix}}$ with $a^2\equiv1$ mod $4e$. We fix a Hodge structure on $\Gamma_{J_a}$ by ${\Gamma_{J_a}}^{2,0}:=T^{2,0}$. Then $\Gamma_{J_a}\cong H^2({S_{I_a}}^{[2]},\Z)$ as Hodge structures, where $I_a:=\Braket{\begin{pmatrix}
\frac{1}{2e}t,\frac{a}{2e}h
\end{pmatrix}}$ and $S_{I_a}$ arises as in the proof of Proposition \ref{prop:pen-ultimo}.%in modo unico!
\end{lemma}
\begin{proof}
Suppose that $z=0$. Let $\pi_1:\T^*\rightarrow A_{\T}$, $\pi_2:(\Z h)^*\rightarrow A_{\Z h}$ and $\pi_3:\Braket{-2}^*\rightarrow A_{\Braket{-2}}$ be the quotient maps. We have $\pi_3^{-1}(0)=\Braket{-2}$.
It follows immediately that $\Gamma_J=(\pi_1\times\pi_2\times\pi_3)^{-1}(J)=(\pi_1\times\pi_2)^{-1}\big(\Braket{\begin{pmatrix}
\frac{1}{2e}t,\frac{a}{2e}h
\end{pmatrix}}\big)\oplus\Braket{-2}$, where $a^2\equiv1$ mod $4e$. We define $I_a:=\Braket{\begin{pmatrix}
\frac{1}{2e}t,\frac{a}{2e}h
\end{pmatrix}}$. As in Proposition~\ref{prop:pen-ultimo}, one can see that there exists only one (up to isomorphism) K3 surface $S_{I_a}$ such that $H^2(S_{I_a},\Z)\cong(\pi_1\times\pi_2)^{-1}(I_a)$ as Hodge structures.
\end{proof}
\begin{rem}\label{rem:conclz=1}
A K3 surface $X\in \FM(S)$ is such that $X\cong {S_{I_a}}$ for a suitable $I_a:=\Braket{\begin{pmatrix}
\frac{1}{2e}t,\frac{a}{2e}h
\end{pmatrix}}$ with $a^2\equiv1$ mod $4e$, by the proof of Proposition~\ref{prop:pen-ultimo}. By Lemma~\ref{lemma:stragrande} and the Generalized Torelli Theorem (see Theorem~\ref{teor:torelligen}), $X^{[2]}$ is a strongly ambiguous Hilbert square if there exists an effective Hodge isometry between $\Gamma_{J_a}$ and $\Gamma_{J_b}$ such that $S_{I_a}\not\cong S_{I_b}$, i.e.\ such that $a\ne\pm b$ (see the proof of Proposition~\ref{prop:pen-ultimo}).
\end{rem}
Now we study when there exists a Hodge isometry between $H^2({S_{I_a}}^{[2]},\Z)$ and $H^2({S_{I_b}}^{[2]},\Z)$ with $a\ne\pm b$.

As $O_{\text{Hodge}}(\T)=\set{\pm\id}$, we want to find for which subgroups $J_a,J_b\subset A_{\T}\times A_{\NS(S^{[2]})}$ we have
\[
(\pm\bar{\id},\bar{\varphi})(J_a)=J_b
\]
for some $\varphi\in O(\NS(S^{[2]}))$ and $a\ne\pm b$. Recall that $\bar\varphi$ means the image of $\varphi$ by the map $O(\NS(S^{[2]}))\rightarrow O(A_{\NS(S^{[2]})})$, the same holds with $T$ instead of $\NS(S^{[2]})$.
\begin{rem}[see~{\cite[Section 3.2]{gal}}]\label{rem:galsuhilbert}
If $e$ is not a square then $O(\NS(S^{[2]}))\cong\Braket{\boldsymbol{\theta},\boldsymbol{\alpha}}\times\set{\pm\id}$, where, on the basis $h,\xi$ of $\NS(S^{[2]})$:
\[
\boldsymbol{\theta}:=\begin{pmatrix}U&V\\eV&U\end{pmatrix},\qquad\boldsymbol{\alpha}:=\begin{pmatrix}1&0\\0&-1\end{pmatrix},
\]
and $(U,V)$ is the minimal positive solution of $\Pell_e(1)$.

If $e$ is a square then $O(\NS(S^{[2]}))=\Braket{\boldsymbol{\alpha}}\times\set{\pm\id}=\set{\pm\boldsymbol{\alpha},\pm\id}$.

A simple computation shows that $\bar{\boldsymbol{\theta}}\in O(A_{\NS(S^{[2]})})$ is an involution.
\end{rem}

The following result gives us a criterium to easily find some cases where there is no strong ambiguity.
\begin{lemma}\label{lemma:esquareamb}
Let $S$ be a K3 surface with Picard number one, such that $\NS(S)=\Z h$ with $h$ ample and $h^2=2e$ for a fixed $e$. If either $e$ is a power of a prime, i.e.\ $p(e)=1$, or $e$ is a square, then $S^{[2]}$ is not strongly ambiguous.
\end{lemma}
\begin{proof}
Suppose that $e$ is a power of a prime. By Proposition~\ref{prop:pen-ultimo} $\#\FM(S)=2^0=1$, hence all the FM partners of $S$ are isomorphic. The claim follows by Theorem \ref{teor:isoFM}.

Suppose that $e$ is a square. By Remark~\ref{rem:galsuhilbert}, $O(\NS(S^{[2]}))=\set{\pm\boldsymbol{\alpha},\pm\id}$.
Hence $(\bar{\id},\pm\bar{\alpha})(J_a)=\Braket{\begin{pmatrix}
\frac{1}{2e}t,\pm\frac{a}{2e}h,0
\end{pmatrix}}=J_{\pm a}=(\bar{\id},\pm\bar{\id})(J_a)$. Then whenever ${S_{I_a}}^{[2]}\cong {S_{I_b}}^{[2]}$ we have $a=\pm b$, and so, by the proof of Proposition~\ref{prop:pen-ultimo}, $S_{I_a}\cong S_{I_b}$.
\end{proof}
We propose here three non trivial examples.
\begin{ex}[$e=6$]\label{ex:e=6}
Let $e=6$. The minimal positive solution of $\Pell_6(1)$ is $(U,V)=(5,2)$. With easy simplifications modulo 12 we know, by Remark~\ref{rem:galsuhilbert}, that
\[
O(\Z h\oplus\Z\xi)=\Braket{\boldsymbol{\theta},\boldsymbol{\alpha}:=\begin{pmatrix}1&0\\0&-1\end{pmatrix}}\times\set{\pm\id},\,\,\boldsymbol{\theta}:=\begin{pmatrix}5&2\\12&5\end{pmatrix},\,\,\bar{\boldsymbol{\theta}}:=\begin{pmatrix}5&2\\0&5\end{pmatrix}.
\]

It follows that
\[
(\bar{\id},\bar{\boldsymbol{\theta}}):J_a:=\Braket{\begin{pmatrix}
\frac{1}{12}t,\frac{a}{12}h,0
\end{pmatrix}}\mapsto J_{5a}:=\Braket{\begin{pmatrix}
\frac{1}{12}t,\frac{5a}{12}h,0
\end{pmatrix}},
\]
with $a^2\equiv1$ mod $24$. Note that $(5a)^2=25a^2\equiv a^2\equiv1$ mod $24$. Moreover $5a\not\equiv\pm a$ mod 24.

Hence the $2^{p(6)-1}=2$ isomorphism classes of FM partners of $S$ define Hilbert squares that could be isomorphic, indeed $(\bar{\id},\bar{\boldsymbol{\theta}})$ induces a Hodge isometry between their second cohomology groups, but we do not know yet if there is an effective one.
\end{ex}
\begin{ex}[$e=10$]\label{ex:e=10}
Let $e=10$. The minimal positive solution of $\Pell_{10}(1)$ is $(U,V)=(19,6)$. With easy simplifications modulo $20$ we know that
\[
O(\Z h\oplus\Z\xi)=\Braket{\boldsymbol{\theta},\boldsymbol{\alpha}:=\begin{pmatrix}1&0\\0&-1\end{pmatrix}}\times\set{\pm\id},\quad\bar{\boldsymbol{\theta}}:=\begin{pmatrix}-1&6\\0&-1\end{pmatrix}.
\]
It is easy to see that
\[
(\bar{\id},\bar{\boldsymbol{\theta}}):J_a:=\Braket{\begin{pmatrix}
\frac{1}{20}t,\frac{a}{20}h,0
\end{pmatrix}}\mapsto J_{-a}:=\Braket{\begin{pmatrix}
\frac{1}{20}t,\frac{-a}{20}h,0
\end{pmatrix}},
\]
with $a^2\equiv1$ mod $40$.
Hence the $2^{p(10)-1}=2$ isomorphism classes of FM partners define Hilbert squares that are not isomorphic. In particular there is not strong ambiguity.
\end{ex}
\begin{ex}[$e=15$]\label{ex:e=15}
Let $e=15$. The minimal positive solution of $\Pell_{15}(1)$ is $(U,V)=(4,1)$. With easy simplifications modulo $30$ we know that
\[
O(\Z h\oplus\Z\xi)=\Braket{\boldsymbol{\theta},\boldsymbol{\alpha}:=\begin{pmatrix}1&0\\0&-1\end{pmatrix}}\times\set{\pm\id},\quad\bar{\boldsymbol{\theta}}:=\begin{pmatrix}4&1\\15&4\end{pmatrix}.
\]
It is easy to see that $J_a$ maps to either $J_{\pm a}$ or an isotropic subgroup with $z=1$. Hence there is not strong ambiguity (see Remark \ref{rem:conclz=1}).
\end{ex}
In Proposition \ref{prop:finalpiccolo} we give a general result on glueing Hodge isometries on $\T$ and $\NS(S^{[2]})$. We need the following technical lemma.
\begin{lemma}\label{lemma:fortunatooo}
Let $(U,V)$ be the (positive) minimal solution of $\Pell_e(1)$. If $V$ is even then $U\not\equiv1$ mod $2e$. Moreover, $V$ is even and $U\equiv-1$ mod $2e$ if and only if $\Pell_e(-1)$ is solvable.
\end{lemma}
\begin{proof}
Suppose that $V$ is even and $U\equiv1$ mod $2e$. Then $V=2m$ and $U=1+2ek$ for suitable $m,k\in\N_{>0}$. Hence:
\[
1=U^2-eV^2=(1+2ek)^2-4em^2=1+4e^2k^2+4ek-4em^2,
\]
which is equivalent to $m^2=k(1+ek)$, thus $k=t^2$ and $1+ek=s^2$ for suitable $t,s\in\N_{>0}$. Then $s^2-et^2=1$, in particular $(U,V)$ is not minimal: a contradiction.

Suppose now that $\Pell_e(-1)$ admits a minimal solution $(s,t)$, then $U=s^2+et^2$ and $V=2st$ is even (see Lemma \ref{rem:pellcolmeno}). Moreover $s^2=-1+et^2$, hence
\[
U=s^2+et^2=-1+2et^2\equiv-1\quad\text{mod $2e$}.
\]

Conversely, suppose that $V$ is even and $U\equiv-1$ mod $2e$. Then $V=2m$ and $U=-1+2ek$ for suitable $m,k\in\N_{>0}$, in particular:
\[
1=U^2-eV^2=(-1+2ek)^2-4em^2=1+4e^2k^2-4ek-4em^2.
\]
This is equivalent to $m^2=k(ek-1)$, in particular $k=t^2$ and $-1+ek=s^2$ for suitable $s,t\in\N_{>0}$. Note that $s^2-et^2=-1+ek-ek=-1$, hence $\Pell_e(-1)$ is solvable.
The claim follows.
\end{proof}
The following result gives necessary and sufficient conditions to have non trivial Hodge isometries between the second cohomology groups of Hilbert squares as in Lemma~\ref{lemma:stragrande}.
\begin{prop}\label{prop:finalpiccolo}
The following are equivalent:
\begin{enumerate}
\item the equation $\Pell_e(1)$ has a positive minimal solution $(U,V)$ with $V$ even and $U\not\equiv-1$ mod $2e$;\label{cond:prop}
\item there exists a (glued) Hodge isometry $\psi:H^2({S_{I_a}}^{[2]},\Z)\rightarrow H^2({S_{I_b}}^{[2]},\Z)$ such that $S_{I_a}\not\cong S_{I_b}$ (i.e.\ $a\ne\pm b$), where $a,b\in\{\alpha\in\Z/2e\Z:\alpha^2\equiv1\text{ mod }4e\}\subset(\Z/2e\Z)^*$.
\end{enumerate}
\end{prop}
\begin{proof}
Let $(U,V)$ be the positive minimal solution of $\Pell_e(1)$, and let $\boldsymbol{\theta}:=\begin{pmatrix}U&V\\eV&U\end{pmatrix}$. We get an isomorphism
\[
(\bar{\id},\bar{\boldsymbol{\theta}}):J_a:=\Braket{\begin{pmatrix}
\frac{1}{2e}t,\frac{a}{2e}h,0
\end{pmatrix}}\mapsto J_{aU}:=\Braket{\begin{pmatrix}
\frac{1}{2e}t,\frac{aU}{2e}h,0
\end{pmatrix}},
\]
which lifts to a Hodge isometry $\psi$ as in the statement by Corollary~\ref{cor:esteniso} and Remark \ref{rem:extensionsquares}. Moreover $a\not\equiv\pm b:=\pm aU$ mod $4e$ by Lemma \ref{lemma:fortunatooo}.

Conversely, let $\psi:H^2({S_{I_a}}^{[2]},\Z)\rightarrow H^2({S_{I_b}}^{[2]},\Z)$ be a Hodge isometry such that $S_{I_a}\not\cong S_{I_b}$ (i.e.\ $a\ne\pm b$). Then there exists an isometry $\boldsymbol{\theta}\in O(\NS(S^{[2]}))$, $\boldsymbol{\theta}\ne\pm\id$. Hence (see also Remark~\ref{rem:galsuhilbert}) there exists a (positive) minimal solution $(U,V)$ of $\Pell_e(1)$ with $V$ even and $U\not\equiv\pm1$ mod $2e$. This proves the equivalence in the statement.
\end{proof}
\begin{lemma}\label{lemma:faticoso}
Suppose that $e$ is not a square and let $(U,V)$ be the (positive) minimal solution of $\Pell_e(1)$. If $e$ is a power of a prime then either $V$ is odd or $\Pell_e(-1)$ is solvable.
\end{lemma}
\begin{proof}
Let $e=p^k$ for suitable $k\in\N$ odd and $p\ne2$ prime. Suppose that $V$ is even, then
\[
U^2=1+p^kV^2\equiv1\quad\text{mod $2p^k$}.
\]
It is known that
\[
\#\Big(\big(\frac{\Z}{2p^k\Z}\big)^*\Big)_2=\#\Big(\big(\frac{\Z}{2\Z}\big)^*\Big)_2\cdot\#\Big(\big(\frac{\Z}{p^k\Z}\big)^*\Big)_2=1\cdot2.
\]
It follows that $U\equiv\pm1$ mod $2p^k$. By Lemma~\ref{lemma:fortunatooo}, if $U\equiv1$ then $V$ is odd, a contradiction, hence $U\equiv-1$, so $\Pell_e(-1)$ is solvable.

Let $e=2^{2k+1}$ for a suitable $k\in\N$. If $k=0$ (i.e.\ $e=2$) the positive minimal solution of $\Pell_2(1)$ is $(U,V)=(3,2)$ with $V$ even and $\Pell_e(-1)$ solvable. We prove by induction that if $k>0$ then $U,V$ are odd. First of all, if $k=1$ (i.e.\ $e=8$) then $(U,V)=(3,1)$. Suppose now that $k>0$ and, by inductive hypothesis, that the positive minimal solution $(\tilde{U},\tilde{V})$ of $\Pell_{2^{2k+1}}(1)$ has $\tilde{U},\tilde{V}$ odd. Let $(U,V)$ be the minimal positive solution of $\Pell_{2^{2(k+1)+1}}(1)$. Notice that
\[
\Pell_{2^{2(k+1)+1}}(1):x^2-2^{2(k+1)+1}y^2=x^2-2^{2k+1}(2y)^2=1,
\]
hence by Lemma \ref{rem:pellcolmeno} $(U,V)$ is obtained by $(\tilde{U}+\tilde{V}\sqrt{2^{2k+1}})^2=\tilde{U}^2+2^{2k+1}\tilde{V}^2+2\tilde{U}\tilde{V}\sqrt{2^{2k+1}}$, i.e.\ $(U,V)=(\tilde{U}^2+2^{2k+1}\tilde{V}^2,\tilde{U}\tilde{V})$. By hypothesis $\tilde{U}$ and $\tilde{V}$ are odd, hence also $U=\tilde{U}^2+2^{2k+1}\tilde{V}^2$ and $V=\tilde{U}\tilde{V}$ are odd.
\end{proof}
\begin{rem}
A necessary condition to the equivalent statements in Proposition~\ref{prop:finalpiccolo} is, by Lemma \ref{lemma:faticoso}, that $e$ is not a square or a power of a prime. In fact, this condition is in Lemma \ref{lemma:esquareamb}.
\end{rem}
%%%%%%%%%%%%%%%%%%%%
\section{Strong ambiguity}\label{sec:strongambiguity}
As before, let $S$ be a K3 surface such that $\NS(S)=\Z h$ with $h$ ample and $h^2=2e$.
Let $\T=\T_S$ be the transcendental lattice of $S$ and let $t\in \T$ such that $A_{\T}=\Braket{\frac{1}{2e}t}$ and $t^2=-h^2$.

In this section we present some results on the effectiveness of the glued Hodge isometries obtained in Section~\ref{sec:speranza}, in order to see when $S^{[2]}$ is strongly ambiguous.
\begin{rem}
Suppose that $\psi:H^2({S_{I_a}}^{[2]},\Z)\rightarrow H^2({S_{I_b}}^{[2]},\Z)$ is a Hodge isometry, where $S_{I_a},S_{I_b}\in\text{FM}(S)$, as in the proof of Proposition \ref{prop:pen-ultimo}, are such that $a\ne\pm b$, and let $C_a,C_b$ be the ample cones (i.e.\ the cones in $\NS({S_{I_a}}^{[2]})\otimes\R$ and $\NS({S_{I_b}}^{[2]})\otimes\R$ generated by the ample divisors) of ${S_{I_a}}^{[2]},{S_{I_b}}^{[2]}$ respectively. It follows from the Generalized Torelli Theorem that there exists an isomorphism $f:{S_{I_b}}^{[2]}\rightarrow {S_{I_a}}^{[2]}$ such that $f^*=\psi$ if and only if $\psi(C_a)=C_b$ (i.e.\ $\psi$ preserves the ample cones).
\end{rem}
We know (see Theorem~\ref{teor:bayermacri}) that the movable cone and the ample cone of $S^{[2]}$ depend on the resolvability of $\Pell_{4e}(5)$.
\begin{lemma}[see also {\cite[Proposition 4.3]{bcns}}]\label{lemma:tecnicmov}
With the notation as in Lemma \ref{rem:galsuhilbert}, the only non trivial isometry in $O(\NS(S^{[2]}))$ that maps the movable cone of $S^{[2]}$ in itself is the involution
\[
\boldsymbol{\beta}:=\boldsymbol{\alpha\theta}=\begin{pmatrix}
1&0\\0&-1
\end{pmatrix}
\begin{pmatrix}
U&V\\eV&U
\end{pmatrix}=
\begin{pmatrix}
U&V\\-eV&-U
\end{pmatrix},
\]
which exists only if $e$ is not a square.
\end{lemma}
\begin{proof}
Let $\mathcal{M}\subset\NS(S^{[2]})\otimes\R$ be the movable cone of $S^{[2]}$.
Notice that every isometry in $O(\NS(S^{[2]}))$ is of the form $\pm\boldsymbol{\theta}^k$ or $\pm\boldsymbol{\alpha\theta}^k$ for suitable $k\in\Z$. By Theorem \ref{teor:bayermacri}, the movable cone has boundary walls $h\R_{>0}$ and $(h-\mu_e\xi)\R_{>0}$, where $\mu_e:=eV/U$ and $(U,V)$ is the minimal positive solution of $\Pell_e(1)$.

First of all, we prove that the cone $C$ with boundary walls $(h+\sqrt{e}\xi)\R_{>0}$ and $(h-\sqrt{e}\xi)\R_{>0}$ is the union of $\{\boldsymbol{\theta}^k(\mathcal{M})\}_{k\in\Z}$, where $\boldsymbol{\theta}^m(\mathcal{M})\cap\boldsymbol{\theta}^n(\mathcal{M})$ is either empty or a boundary wall. To show this claim, notice that the map $\boldsymbol{\theta}$ has eigenvectors $h\pm\sqrt{e}\xi$ and $\boldsymbol{\theta}(C)=C$. Moreover
\begin{align*}
\boldsymbol{\theta}&:(h-\mu_e\xi)\R_{>0}\mapsto h\R_{>0},\\
\boldsymbol{\theta}&:h\R_{>0}\mapsto (h+\mu_e\xi)\R_{>0},\\
\boldsymbol{\theta}&:(h+\mu_e\xi)\R_{>0}\mapsto (h+c\xi)\R_{>0},
\end{align*}
where $\mu_e<c<\sqrt{e}$. Indeed, if $c\le\mu_e$ then $\boldsymbol{\theta}$ has a third eigenvector, a contradiction. The claim follows by iterating the previous argument. Hence $\pm\boldsymbol{\theta}^k$ does not preserve $\mathcal{M}$ for all $k\in\Z\backslash\{0\}$.

The lemma follows by the fact that $\boldsymbol{\alpha}$ is the reflection with respect to $h$. Indeed, with a simple computation one can see that
\begin{align*}
\boldsymbol{\beta}&:h\R_{>0}\mapsto (h-\nu_e\xi)\R_{>0},\\
\boldsymbol{\beta}&:(h-\nu_e\xi)\R_{>0}\mapsto h\R_{>0}.\qedhere
\end{align*}
\end{proof}
The following theorem gives us necessary and sufficient conditions to have strong ambiguity of Hilbert squares of projective K3 surfaces with Picard number one. It is equivalent to \cite[Proposition 3.14]{dm}, see Remark \ref{rem:equivdm}.
\begin{teor}\label{teor:analogodm}
Let $S$ be a K3 surface such that $\NS(S)=\Z h$ with $h$ ample and $h^2=2e$.
The Hilbert square $S^{[2]}$ is strongly ambiguous if and only if:
\begin{enumerate}
\item the Pell-type equation $\Pell_e(1)$ has a (positive) minimal solution $(U,V)$ with $V$ even and $U\not\equiv-1$ mod $2e$;
\item the Pell-type equation $\Pell_{4e}(5)$ is not solvable;
\end{enumerate}
\end{teor}
\begin{proof}
Let $S=S_{I_a}$ for a suitable $a\in(\Z/2e\Z)^*$ such that $a^2\equiv1$ mod $4e$, as in the proof of Proposition \ref{prop:pen-ultimo}.
The Hilbert square ${S_{I_a}}^{[2]}$ is strongly ambiguous if and only if there exists a Hodge isometry $\psi:H^2({S_{I_a}}^{[2]},\Z)\rightarrow H^2({S_{I_b}}^{[2]},\Z)$ such that $S_{I_a}\not\cong S_{I_b}$ (i.e.\ $a\ne\pm b$) where $b\in(\Z/2e\Z)^*$ and $b^2\equiv1$ mod $4e$, as in Proposition \ref{prop:finalpiccolo}, and moreover $\psi$ is effective. By Lemma \ref{lemma:tecnicmov} the only isometry in $O(\NS({S_{I_a}}^{[2]}))$ that preserves the movable cone is $\boldsymbol{\beta}:=\boldsymbol{\alpha\theta}$. Hence ${S_{I_a}}^{[2]}$ is strongly ambiguous if and only if the Hodge isometry $\psi$ obtained by glueing ${\id}$ with ${\boldsymbol{\beta}}$ is effective.

We prove now that this isometry is effective if and only if $\Pell_{4e}(5)$ is not solvable. If $\Pell_{4e}(5)$ is not solvable then by Theorem \ref{teor:bayermacri} the movable cone is equal to the ample cone, so $\psi$ is effective. Conversely, suppose that $\Pell_{4e}(5)$ is solvable; then by Theorem \ref{teor:bayermacri} the boundary walls of the ample cone are $h\R_{>0}$ and $(h-\nu_e\xi)\R_{>0}$, where $\nu_e:=2e\frac{b_5}{a_5}$ with $(a_5,b_5)$ the positive minimal solution of $\Pell_{4e}(5)$ and $\nu_e<\mu_e$, in particular the ample cone is contained in the movable cone. Notice that
\begin{align*}
\boldsymbol{\beta}&:h\R_{>0}\mapsto (h-\mu_e\xi)\R_{>0},\\
\boldsymbol{\beta}&:(h-\nu_e\xi)\R_{>0}\mapsto (h-c\xi)\R_{>0},
\end{align*}
where $0<c<\mu_e$. Suppose that $0<c<\nu_e$, then $\psi$ is the pullback of an isomorphism between the Hilbert squares ${S_a}^{[2]}$ and ${S_b}^{[2]}$ by the Generalized Torelli Theorem. But not every ample divisor maps to ample divisor, a contradiction. Hence $\nu_e\le c<\mu_e$ and $\psi$ can not be effective.
\end{proof}
\begin{rem}\label{rem:equivdm}
By Lemma \ref{lemma:fortunatooo}, it is trivial that Theorem \ref{teor:analogodm} is equivalent to \cite[Proposition 3.14]{dm}, but here it is proved from a different point of view.
\end{rem}
\begin{cor}
In Example \ref{ex:e=6} there is strong ambiguity, in Examples \ref{ex:e=10}, \ref{ex:e=15} there is no strong ambiguity.
\end{cor}
\begin{proof}
The claim follows easily by Theorem~\ref{teor:analogodm}. In the Example~\ref{ex:e=10} the equation $\Pell_e(-1)$ is solvable, take for instance $(3,1)$. In the Example~\ref{ex:e=15} the equation $\Pell_e(1)$ has minimal solution $(U,V)=(4,1)$ and $V$ is odd.
\end{proof}
%%%%%%%%%%%%%%%%%%%%%%%
\section{Automorphisms of Hilbert squares}\label{sec:Automorphisms}
We now apply the previous algorithm to compute $\Aut(S^{[2]})$. The following result is \cite[Theorem 1.1]{bcns}, but proved from another point of view.
\begin{teor}\label{teor:bcns}
Let $S$ be a K3 surface with Picard number one such that $\NS(S)=\Z h$ with $h$ ample and $h^2=2e$.
The automorphism group $\Aut(S^{[2]})$ is either trivial or generated by a non trivial involution. In particular $\Aut(S^{[2]})$ is not trivial if and only if:
\begin{enumerate}
\item e is not a square;
\item the Pell-type equation $\Pell_e(-1)$ is solvable;
\item the Pell-type equation $\Pell_{4e}(5)$ is not solvable.
\end{enumerate}
\end{teor}
\begin{proof}
Let $S=S_{I_a}$ for a suitable $a\in(\Z/2e\Z)^*$ such that $a^2\equiv1$ mod $4e$, as in the proof of Proposition \ref{prop:pen-ultimo}.
Then the group $\Aut(S^{[2]})$ is in bijection with the group
\[
\G:=\{\psi:H^2({S_{I_a}}^{[2]},\Z)\rightarrow H^2({S_{I_a}}^{[2]},\Z)\,|\,\psi\text{ is an effective Hodge isometry}\}.
\]

A necessary condition to have a non trivial $\psi\in\G$ is that $e$ is not a square (see Remark \ref{rem:galsuhilbert}). Under this hypothesis, by Lemma \ref{lemma:tecnicmov} the only non trivial isometry in $O(\NS(S^{[2]}))$ that maps the movable cone of $S^{[2]}$ in itself is the involution $\boldsymbol{\beta}:=\boldsymbol{\alpha\theta}$. Let $(U,V)$ be the minimal positive solution of $\Pell_e(1)$. We need $V$ even, hence by Lemma \ref{lemma:fortunatooo} $U\not\equiv1$ mod $2e$. Then we have to impose $U\equiv-1$ mod $2e$, otherwise $(\pm\bar{\id},\bar{\boldsymbol{\beta}})$ does not preserve $J_a$. Again by Lemma \ref{lemma:fortunatooo}, $U\equiv-1$ mod $2e$ and $V$ even is equivalent to $\Pell_e(-1)$ solvable. To preserve $J_a$, the only possibility is $\psi$ obtained by glueing $-\id$ with $\boldsymbol{\beta}$, indeed
\[
(-\bar{\id},\bar{\boldsymbol{\beta}}):J_a:=\Braket{
\Big(\frac{1}{2e}t,\frac{a}{2e}h,0\Big)
}\mapsto \Braket{
\Big(\frac{-1}{2e}t,\frac{-a}{2e}h,0\Big)
}=J_a.
\]
In this case, as in the proof of Theorem \ref{teor:analogodm}, $\psi$ is effective if and only if $\Pell_{4e}(5)$ is not solvable.
\end{proof}
%%%%% BIBLIOGRAPHY %%%%%%%%%%%%%%

\end{document}